\documentclass[11pt]{amsart}
\usepackage{amsmath}
\usepackage{amssymb}
\usepackage{tabularx}
\usepackage{enumerate}
\usepackage{graphicx}
\usepackage{texdraw}
\usepackage{mathrsfs}
\usepackage{epstopdf}
\usepackage[normal]{subfigure}
\usepackage{float}
\usepackage[pagewise]{lineno}
\usepackage{soul}

\usepackage{mathrsfs}
\usepackage{amsfonts,amssymb,amsmath}
\usepackage{epsfig}
\usepackage{bm}
\usepackage{color}

\makeatletter
\@addtoreset{equation}{section}

\makeatother
\newtheorem{thm}{Theorem}[section]
\newtheorem{lem}[thm]{Lemma}

\newtheorem{prop}[thm]{Proposition}
\newtheorem{remark}[thm]{Remark}

\newcommand{\R}{\mathbb{R}}

\begin{document}
\title[The effect of an unfavorable region]{The effect of an unfavorable region on the invasion process of a species$^\S$}
\thanks{$\S$ This research was partly supported by the NSFC (No. 12101413, 12071299), and the NSF of Shanghai (20JC1413800).}
\author[P. Lai and J. Lu]{Pengchao Lai$^\dag$ and Junfan Lu$^{\dag, *}$}
\thanks{$\dag$ Mathematics and Science College, Shanghai Normal University, Shanghai 200234, China.}
\thanks{$*$ Corresponding author.}
\thanks{{\bf Emails:} {\sf laipchao@163.com} (P. Lai), {\sf jlu@shnu.edu.cn} (J. Lu)}
\date{}

\baselineskip 15pt

\begin{abstract}
To model a propagating phenomena through the environment with an unfavorable region, we consider a reaction diffusion equation with negative growth rate in the unfavorable region and bistable reaction outside of it. We study rigorously the influence of $L$, the width of the unfavorable region, on the propagation of solutions.
It turns out that there exists a critical value $L^*$ depending only on the reaction term such that, when $L<L^*$, spreading happens for any solution in the sense that it passes through the unfavorable region successfully and establish with minor defect in the region;
when $L=L^*$, spreading happens only for a species with large initial population, while residue happens for a population with small initial data, in the sense that the solution converges to a small steady state;
when $L>L^*$ we have a trichotomy result: spreading/residue happens for a species with large/small initial population, but, for a species with medium-sized initial data, it can not pass through the region either and converges to a transition steady state.
\end{abstract}

\subjclass[2010]{Primary 35B40, 92D25; Secondary 35K57}
\keywords{Population dynamics; reaction diffusion equation; unfavorable region; asymptotic behavior.}
\maketitle

\section{Introduction}
In the field of mathematical biology, reaction diffusion equations like
\begin{equation}\label{RDE}
u_t = \Delta u +f(u)
\end{equation}
are often used to describe the dynamics and the spreading phenomena of a species, where $f(u)/u$ represents the growth rate of the species. Typical examples of $f$ include the so-called monostable case like $f(u)=u(1-u)$ and bistable case like $f(u)=u(u-\frac13) (1-u)$. These equations have been studied systematically in the last decades. For example, in the monostable case, Aronson and Weinberger \cite{AW1,AW2} proved a {\it hair-trigger effect}, which says that {\it spreading} (also called {\it persistence} or {\it propagation}: $u$ converges to a positive stationary solution) happens for any positive solution. In the bistable case, \cite{AW1,AW2} gave sufficient conditions for {\it spreading} and that for {\it vanishing} (also called {\it extinction}: $u\to 0$ as $t\to \infty$). In 1977, Fife and McLeod \cite{FM} proved the existence and stability for the traveling wave solution of the bistable equation. In 2006, Zlato\v{s} \cite{Zla} gave a complete study on the bistable equation, and proved a trichotomy result on the asymptotic behavior for the solutions: there exists a sharp value $L^*>0$ such that when $L>L^*$ (resp. $L<L^*$), spreading (resp. vanishing) happens for the solution of the bistable equation with initial data $u(x,0)=\chi_{[-L,L]}(x)$; when $L=L_*$, the solution develops to a transition  ground state solution.
In 2010, Du and Matano \cite{DM} extended these results to the Cauchy problem with general initial data like $u(x,0) =\phi_\lambda (x)$. For bistable equations, they also proved the trichotomy and sharp transition results as in \cite{Zla}.

In practical problems, however, a spreading process is generally not smooth due to the spatial and/or temporal heterogeneity. A favorable environment can accumulate the spreading while an unfavorable one will slow down or even block the spreading. For example, in the invasion process of a harmful species or infectious diseases, people often establish isolation zones to prevent the invasion by intensive elimination, or in the spreading process of a beneficial species, it may encounter adverse environments such as deserts, swamps, environmental pollution, etc..
All of these unfavorable (for the species) environments can be regarded as desert regions for the propagation.  (Note that, an unfavorable region we consider here is different from an obstacle area if the latter one is regarded as a region without population at all).
In both ecology and mathematics, whether the species can pass through the region successfully or not is a problem that people are interested in (e.g., \cite{LWZ,LZC,ML,ZKL} etc.).

From mathematical point of view, if the spreading process of a species obeys the reaction diffusion equation as \eqref{RDE} before it encounters an unfavorable region, then the reaction term should be negative to the growth of the population when the species enters the unfavorable region, such as
$f= -k u + \varepsilon (u)$ for some small $\varepsilon(u)$ (e.g., \cite{Wang1} etc.), or just $f=-ku$ for simplicity. Intuitively, a species can passes though the unfavorable region easily when $L$ (the width of the unfavorable region) and $k$ (the death rate of the species in the unfavorable region) are small, but not when $L, k$ are large. In this paper, we normalize the parameter $k=1$, and study rigorously the influence of $L$ on the propagation phenomena in one space dimension. We will show that there is a critical number $L^*$ such that, when $L<L^*$, {\it spreading happens} for a species in the sense that it passes through the unfavorable region successfully and establish with minor defect in the region; when $L=L^*$, we have a dichotomy result: spreading happens only for a species with large initial population, while {\it residue happens} for a species with small initial data, in the sense that the solution converges to a small steady state, which takes small value in the front area; when $L>L^*$, we have a trichotomy result: spreading or residue happens for a species with large or small initial population. For a species with medium-sized initial data, it can not pass through the region either, but converges to a transition steady state which lies between the final states of spreading and residue.

Our mathematical model is the following Cauchy problem:
\begin{equation}\label{main-eq}
 \left\{\begin{array}{ll}
  u_t = u_{xx} + f(x,u), ~& x\in \mathbb{R},\ t>0,\\
  u(x,0)= u_0(x),  ~&x\in \mathbb{R},
        \end{array}
        \tag{P}
 \right.
\end{equation}
where the reaction term $f(x,u)$ is defined as the following:
\begin{equation}
\label{condition1}
\tag{H}
f(x,u)=
\begin{cases}
u (u-\alpha )(1-u), & |x| \geq L,\\
- u , & |x| <  L.
\end{cases}
\end{equation}
Here, $\alpha\in (0,\frac12)$, and as we mentioned above, $2L$ is used to denote the width of the unfavorable region $[-L,L]$, in which the environment is unfavorable for the species so that the growth rate is $-1$. We will see below that, as can be expected, the propagation will be more difficult than the original bistable equation, due to the destructiveness in the unfavorable region. This kind of model was also used by some authors to understand the phenomena of crime (e.g., \cite{BN, BRR} and references therein). The unknown $u(x,t)$ then denotes the population's propensity to commit a crime, and the unfavorable region $[-L,L]$ then denotes a prevention area of the invasions of crimes.
Our background is different from theirs, but the mechanism is similar. As can be seen below, our results focus more on the convergence of solutions than theirs (see more in Remark \ref{rem:to-Bere-1}).

We take initial data from $L^\infty(\R)$, then the well-posedness of $W^{2,1}_{p, loc}$ strong solution follows from the standard parabolic theory. We are then interested in the qualitative property of the solutions. We will show that any global solution of \eqref{main-eq} will converge as $t\to \infty$ to a stationary one, as in the typical reaction diffusion equations (e.g., \cite{DL, DM, Zla}).
Since we are studying a species with sufficient large population in the original habitat, and it tries to propagate further through an unfavorable region, it is natural to consider the initial data as the following:
\begin{equation}\label{ini-cond-2}
\tag{I}
\left\{
 \begin{array}{l}
 u_0(x)\in L^\infty(\R), \mbox{ it is decreasing in }\R \mbox{ with } u_0(-\infty)=1,\ u_0(\infty)=0,\\
 \mbox{and } \displaystyle \lim\limits_{x\to -\infty} \frac{|1-u_0(x)|}{e^{\sqrt{1-\alpha}\; x}} =0 \mbox{ or } \infty,\quad  \lim\limits_{x\to +\infty} \frac{u_0(x)}{ e^{-\sqrt{\alpha}\; x}} =0 \mbox{ or }\infty.
 \end{array}
 \right.
\end{equation}
The decay rates as $x\to \pm \infty$ are imposed only for technical reasons when we use the zero number argument to prove a general convergence result. They are not essential and can be weakened (see more in Section 3 Remark \ref{rem:weaken-ini-cond}).
Clearly, the typical examples: $H(a-x),\ U(x-ct-a)$ satisfy these conditions,  where $a\in \R$, $H$ is the Heaviside function and $U(x-ct)$ is the traveling wave solution of the bistable equation with $U'(z)<0$.

Since the classification of the positive stationary solutions depends on the size of $L$, before we state our main theorems, it is necessary to specify the stationary solutions of (P) clearly. Stationary solutions solve the following equation:
\begin{equation}\label{equ1.1}
v''+f(x,v)=0,\quad x \in \mathbb{R}.
\end{equation}
By a careful phase plane analysis (see details in Section 2), we will see that there exists a critical value $L^* >0$ (which depends only on $f$), such that when $L\in (0,L^*)$, \eqref{equ1.1} has exactly one positive stationary solution, denoted by $\mathcal{V}_b(x)\in W^2_{p, loc}(\R)$ for any $p>1$, which satisfies
\begin{equation}\label{V-b-property}
\mathcal{V}_b(-x)=\mathcal{V}_b(x),\ \  \mathcal{V}'_b(x)>0 \mbox{ for } x >0,\quad \mathcal{V}_b(0)>0,\quad \mathcal{V}_b(\pm \infty)=1.
\end{equation}
For convenience, we call $\mathcal{V}_b$ as a {\bf big stationary solution}, and say {\bf spreading happens} for a solution $u$ of the problem \eqref{main-eq} if
$$
\lim\limits_{t\to \infty} u(x,t)= \mathcal{V}_b(x),\mbox{\ \ in the topology of }L^\infty_{loc}(\R).
$$
Then in case $0<L<L^*$ we have the following convergence result.

\begin{thm}[Spreading result for small $L$]\label{thm1.1}
Assume $L \in(0,L^*)$. Then for any $u_0$ satisfying   \eqref{ini-cond-2}, spreading happens for the solution $u$ of \eqref{main-eq}.
\end{thm}

In case $L=L^*$, the equation \eqref{equ1.1} has two positive stationary solutions: $\mathcal{V}_b$ as above and $\mathcal{V}_s \in W^2_{p,loc}(\R)$ for any $p>1$, the latter satisfies
\begin{equation}\label{V-s-property}
\mathcal{V}_s(-\infty)=1,\quad \mathcal{V}_s(+\infty)=0,\quad \mathcal{V}'_s (x)<0 \; \mbox{for} \; x \in \mathbb{R}.
\end{equation}
For convenience, we call $\mathcal{V}_s$ as a {\bf small stationary solution}, and say
that {\bf residue happens} for the solution $u$ if it converges to $\mathcal{V}_s(x)$ as $t\to \infty$ in the topology of $L^\infty_{loc}(\R)$.

\begin{thm}[Dichotomy result for critical $L^*$]\label{thm1.2}
 Assume $L=L^*$ and $\phi$ satisfies   \eqref{ini-cond-2}. For any $\sigma\in \R$, let $u_\sigma(x,t)$ be the time-global solution of \eqref{main-eq} with initial data $u_0 = \phi (x-\sigma)$. Then either spreading happens as above for all $\sigma \in \R$, or there exists $\sigma^* \in \R$ such that
\begin{itemize}
\item[(i)] spreading happens when $\sigma \in (\sigma^*,+\infty)$;
\item[(ii)] residue happens when $\sigma \in (-\infty,\sigma^*]$.
\end{itemize}
\end{thm}

When $L> L^*$, the equation \eqref{equ1.1} has exactly three positive stationary solutions: the big solution $\mathcal{V}_b$ and the small solution $\mathcal{V}_s$ as above, and another solution $\mathcal{V}_g \in W^2_{p, loc}(\R)$ for any $p>1$ which lies between $\mathcal{V}_s$ and $\mathcal{V}_b$, and
\begin{equation}\label{V-g-property}
 \left\{
 \begin{array}{l}
 \mathcal{V}_g(-\infty)=1,\quad \mathcal{V}_g(+\infty)=0,\quad \mathcal{V}_s(x) < \mathcal{V}_g(x)< \mathcal{V}_b(x) \mbox{ for all }x\in \R,\\
 \mathcal{V}_g \mbox{ has a unique local minimum }x_1 \in (-L,L) \mbox{ and a unique local maximum }x_2 >L.
 \end{array}
 \right.
\end{equation}
We call $\mathcal{V}_g$ a {\bf transition (stationary) solution}, and say $u$ is a {\bf transition solution} of (P) if it converges to $\mathcal{V}_g$ in the topology of $L^\infty_{loc}(\R)$.
As we expect before, when $L>L^*$, the propagation is heavily effected by the unfavorable region and so residue happens easier. Nevertheless, we have complicated situations for the asymptotic behavior of the solutions in this case, that is, all of residue, transition and spreading are possible.

\begin{thm}[Trichotomy result for large $L$]\label{thm1.3}
Assume $L >L^*$ and $\phi$  satisfies   \eqref{ini-cond-2}. For any $\sigma\in \R$, let $u_\sigma (x,t)$ be a time-global solution of \eqref{main-eq} with initial data $u_0 = \phi(x-\sigma)$. Then there exist $-\infty< \sigma_* \leq \sigma^* <\infty$ such that
\begin{itemize}
\item[(i)] spreading happens when $\sigma \in (\sigma^*,+\infty)$;
\item[(ii)] residue happens when $\sigma \in (-\infty,\sigma_*)$;
\item[(iii)] $u_\sigma$ is a transition solution when $\sigma \in [\sigma_*,\sigma^*]$.
\end{itemize}
\end{thm}

By these results we know that, from a ecological point of view, an unfavorable region
will definitely affect the propagation quantitatively, no matter how large the region is.
On the other hand, whether the propagation is affected qualitatively depends on the size of the region. More precisely, when the width $2L$ of the region is smaller than the critical size $2L^*$ depending on $f$, the effect is a quantitative one: propagation is always  successful but with a defective limit, that is, the limit $\mathcal{V}_b<1$. When the width is larger than the critical value, the effect can be a qualitative one: small initial population may be blocked and can not propagate through  the region successfully. As a result, it converges to a residue or transition stationary solution. Note that, these solutions remain positive on the right side with small value just because the problem is a Cauchy one, but not means the species can pass through the unfavorable region smoothly.

For the clarity of our statements, in the previous theorems we adopted a cubic bistable nonlinearity and a constant death rate, outside of and in the unfavorable region, respectively. Of course, one can consider more general reaction terms. For example, $f(x,u)$ is a general bistable nonlinearity $f_b(u)$ outside of the unfavorable region and
a general negative growth rate $f_m(u)<0\ (u>0)$ in the unfavorable region.
In this case, analogue of our current conclusions can be derived in a similar way.
More precisely, there are two positive real numbers $L^* \geq L_*>0$ depending only on $f$ such that the asymptotic behavior for the solutions of \eqref{main-eq} are as the following.
\begin{enumerate}[(1).]
\item When $L<L_*$, the equation \eqref{main-eq} has only positive stationary solutions of $\mathcal{V}_b$ type (maybe not unique), spreading happens in the sense that each solution $u$ converges to one of them.

\item When $L\in [L_*, L^*)$ (or $L=L^*$ in case $L_* = L^*$ as in our current case), the equation \eqref{main-eq} has positive stationary solutions of both $\mathcal{V}_b$ and $\mathcal{V}_s$ types (maybe not unique), and a dichotomy results hold: a solution with large (resp. small) initial data converges to a stationary solution of the type $\mathcal{V}_b$ (resp. $\mathcal{V}_s$).

\item When $L\geq L^*$ (or $L>L^*$ in case $L_* = L^*$ as in our current case),  \eqref{main-eq} has positive stationary solutions of $\mathcal{V}_b, \mathcal{V}_g$ and $\mathcal{V}_s$ types (each type maybe not unique), a trichotomy results hold: a solution with large (resp. small, medium-sized) initial data converges to a stationary solution of the type $\mathcal{V}_b$ (resp. $\mathcal{V}_s$, $\mathcal{V}_g$).
\end{enumerate}
The proof of these results is similar as ours, and the details will be given in a forthcoming paper (see more in Remarks \ref{rem:general-bi-ss2}).

\begin{remark}\rm
If the reaction term is a monostable one like $f(u)=u(1-u)$, there is the so-called hair-trigger effect (e.g., \cite{AW1,AW2}), which says that any nonnegative, non-trivial solution will converge to $1$, no matter how small the initial data is. Hence, if we use monostable model instead of the bistable one outside of the unfavorable region, then the asymptotic behavior will be very simple: spreading always happens.
\end{remark}

\begin{remark}\label{rem:to-Bere-1}
\rm
As we have mentioned above, Berestycki et al. \cite{BRR} also studied the equation \eqref{main-eq} as a model to understand the phenomena of crime. They proved that, when $L<L_*$, the wave propagates in the sense that $u(x,t)>1-\varepsilon$ for any small $\varepsilon>0$ and all large $x,t$; when $L\geq L^*$, the propagation is blocked in the sense that there exists a stationary solution as our $\mathcal{V}_s$. We see that these results correspond to our spreading and residue phenomena, but without the convergence to the corresponding stationary solutions. Moreover, we complete the dichotomy (in case $L=L^*$) and trichotomy (in case $L>L^*$) results, as well as the threshold for the initial data as in Theorem \ref{thm1.3}.
\end{remark}

The rest of the paper is organized as follows. In Section 2, we present the positive stationary solutions by using a phase plane analysis. In Section 3, we give a general convergence result for the solutions of \eqref{main-eq} with initial data satisfying \eqref{ini-cond-2}. In Section 4, we prove our main theorems on the dichotomy and  trichotomy results.

\section{Positive Stationary Solutions}

In this section we will present the definition of the critical value $L^*$ of the unfavorable region width, and construct some useful positive stationary solutions by using the phase plane analysis.

\subsection{Phase plane}
We will work on the equation of stationary solutions:
\begin{equation}\label{v-eq}
v'' +f(x, v)=0,\quad x \in \mathbb{R},
\end{equation}
for $f$ defined by (H). For convenience, in what follows we write
$$
f_l (u) = f_r(u) := u (u-\alpha) (1-u),\quad f_m (u):= -u,
$$
then $f(x,u)=f_l(u) =f_r(u)$ when $|x|\geq L$ and $f(x,u)=f_m(u)$ when $|x|<L$. Each solution of equation \eqref{v-eq} consists of three parts
 \begin{equation*}
v=
\begin{cases}
v_{l}(x), & x \leq L,\\
v_{m}(x), &|x| < L,\\
v_{r}(x), & x \geq L,
\end{cases}
\end{equation*}
which satisfies the following equations
\begin{align}
v''_{l}+ f_l (v_l)=0, \ \ \ \ \ x < -L,\ \label{vlequ}\\
v''_{m}+ f_m (v_m)=0, \ \  |x| < L, \  \label{vmequ}\\
v''_{r}+ f_r(v_r)=0, \ \ \ \ x > L, \ \ \ \label{vrequ}
\end{align}
with matching conditions on the boundaries of the unfavorable region:
\begin{equation}\label{match-cond}
\left\{
 \begin{array}{ll}
v_{l}(-L-0)=v_{m}(-L+0),& v'_{l}(-L-0)=v'_{m}(-L+0),\\
v_{m}(L-0)=v_{r}(L+0), & v'_{m}(L-0)=v'_{r}(L+0).
\end{array}
\right.
\end{equation}
Now, we seek for solutions of the problems \eqref{vlequ}-\eqref{vrequ} by phase plane analysis. First, these equations can be converted into the following systems:
\begin{equation}\label{odes2}
\begin{cases}
v'_{i}= w_{i}, \\
w'_{i}=-f_{i}(v_{i}),
\end{cases}
\end{equation}
for $i=\{l,m,r\}$. These systems are equivalent to the following first-order ordinary differential equations
\begin{equation}\label{ode1}
w_i d w_{i} = - f_{i}(v_{i}) dv_i.
\end{equation}
They are solved explicitly:
\begin{equation}\label{ode1sol}
w^{2}_{i}= C- 2\int^{v_{i}}_0 f_{i}(s)ds,
\end{equation}
for some $C$ to be determined. Taking different $C$ we can sketch the phase diagram of \eqref{vlequ}-\eqref{vrequ} or \eqref{odes2} as in Figure 1. For convenience, we present some details here.
\begin{figure}[H]
\centering
\includegraphics[scale=0.5]{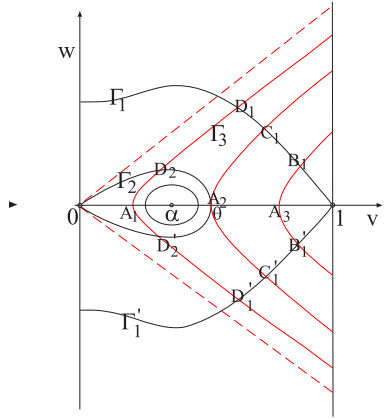}
\caption{Trajectories of the system \eqref{odes2}.}
\end{figure}
In Figure 1, the trajectory $\Gamma_1$ is the graph of \eqref{ode1sol} with
$i=l$ and $C=2 \int_0^1 f_l(s) ds>0$. It corresponds to a solution $V_1$ of \eqref{vlequ} or \eqref{vrequ}, which takes $0$ at some point $x_0$ and is strictly increasing for $x>x_0$ with
$$
0< 1- V_1(x) \sim e^{-\sqrt{1-\alpha}\; x},\quad x\to \infty.
$$
Similarly, the trajectory $\Gamma'_1$ corresponds to a solution $\tilde{V}_1(x) \equiv V_1(2 x_0 -x)$.

$\Gamma_2$ is the graph of \eqref{ode1sol} with $i=l$ and $C=0$. It is a homoclinic orbit, and corresponds to the so-called {\it ground state solution} $V_2$ of \eqref{vlequ} or \eqref{vrequ}, which satisfies
$$
V_{2}(x)= V_{2}(-x) \mbox{ and } V'_{2} (x)<0 \mbox{ for } x>0,\quad
 V_{2}(0)= \theta, \quad V_2 (x) \sim e^{-\sqrt{\alpha}\; x}\mbox{ as } x\to +\infty,
$$
where $\theta = \theta(\alpha)\in (\alpha, 1)$ is defined by $
\int^{\theta}_0 f_l(u)du = \int^{\theta}_0 f_r(u)du =0$, that is,
$$
\theta = \theta (\alpha) := \frac{4(\alpha+1)-\sqrt{16(\alpha+1)^2-72\alpha}}{6}.
$$

The trajectory $\Gamma_3$ is the graph of \eqref{ode1sol} with $i=l$ and $C=2\int_0^a f_l(s) ds$ for any given $a \in (\theta,1)$. 
It gives a compactly supported stationary solution $V_a$ of the bistable reaction diffusion equation, which satisfies, for some $x_0\in (0,\infty)$ depending on $a$
\begin{equation}\label{ss-bistable}
V_a(\pm x_0)=0,\quad V_a (0)=q_0,\quad V_a (-x)=V_a (x) \mbox{ and }V'_a(x)<0 \mbox{ in } (0,x_0].
\end{equation}
Such solutions will be used in Section 3 to give a sufficient condition for the spreading phenomena.

The singular point $(\alpha, 0)$ is a center. For any $a\in (0,\alpha)$, there is a closed orbit $\Gamma_4$ surrounding $(\alpha,0)$, which is the graph of \eqref{ode1sol} with $i=l$ and $C=2\int_0^a f_l(s) ds$, and corresponds to a periodic solution, denoted by $V_a^{per}$.

For any small $\delta>0$, the trajectory $\Gamma^*_5$ passing through $F(1,-\delta)$ is the graph of \eqref{ode1sol} with $i=l$ and $C=\delta^2 + 2\int_0^1 f_l(s) ds$. It corresponds to a solution $V^*_5$ which satisfies
$$
V^*_5(x)<0 \mbox{ for } x<0 ,\quad V^*_5(x)\to +\infty \mbox{ as }x\to -\infty,\quad V^*_5(0)=0 \mbox{ and }V^*_5 (x_0) =1,
$$
for some $x_0<0$ (here we normalize the function as zero at $x=0$).

Finally, $\Gamma_0$ is a trajectory of the system \eqref{odes2} for $i=m$. For any $a\in (0,1)$, there is a $\Gamma_0$ passing through $A(a,0)$, which is the graph of the function $w^2_{m}=v^2_{m}-a^2\;(v_{m} \geq a)$. This trajectory gives a solution $V_0$ of \eqref{v-eq} in the interval $[-L,L]$, and can be expressed explicitly as
\begin{equation}\label{def-vm}
   V_0 (x; a)= a \cosh(x) \geq a.
\end{equation}

\subsection{Trajectory pieces corresponding to the solutions on $[-L,L]$}

We will combine suitable trajectory pieces of \eqref{vlequ}-\eqref{vrequ} together to give solutions of the problem \eqref{v-eq}. Note that each one of such solutions satisfies \eqref{vmequ} in a domain with exact width $2L$. So we need to specify the relationship between the arc lengths of the trajectory pieces and the life spans of corresponding solutions.

We introduce a notation. If $AB$ is a trajectory piece in the phase plane corresponding to a solution $V(x)$ of \eqref{v-eq} such that
$$
(V(x_1), V'(x_1)) =A, \quad (V(x_2), V'(x_2)) =B,
$$
then $|x_1 -x_2|$ is the life span of the solution given by this piece $AB$. In what follows we use $X_{AB}$ to denote this span $|x_1 -x_2|$.

In Figure 1, we see that if $A=(a,0)\in \Gamma_0$ lies between $(0,0)$ and $(\theta,0)$, then $\Gamma_2$ and $\Gamma_0$ has exactly two intersection points $D_2,\ D'_2$; if $a = \theta $, then there is a unique intersection point $A=(\theta,0)$ between them; if $a \in (\theta,1)$,  there is no intersection point between them. On the other hand, any $\Gamma_0$ has a unique intersection point $D_1$ with $\Gamma_1$.
Denote $D_i = (v_i, w_i)\ (i=1,2)$. Since the stationary solution corresponding to $\Gamma_0$ is given by $V_0$ in \eqref{def-vm}, we see that
\begin{equation}\label{def-a-r-R}
V_0 (0)=a,\quad V_0 (r)= v_2,\quad  V_0 (R) = v_1,
\end{equation}
for some $R= R(a) >r= r(a) >0$, that is, $X_{AD_2}=r$ and $X_{AD_1}=R$.

We now study the monotonicity of $R(a),\ r(a)$ and $\ell(a):=R(a)-r(a)= X_{D_1 D_2}$.

\begin{lem}\label{lem2.1}
All of $R, r$ and $\ell$ given above are strictly decreasing in the parameter $a\in (0,1)$.
\end{lem}

\begin{proof}
(1). We first prove $\frac{dR}{da}<0$. Recall that the function of $\Gamma_0$ is
$w^2  = v^2 -a^2$, and the function of $\Gamma_1$ is $w^2 = 2\int_{v}^1 f_l(s) ds$.
Combining them together we obtain the $v$-coordinate $v_1$ for $D_1$:
\begin{equation}\label{equ 2.5}
v^2_1 -a^2 = 2\int_{v_1}^1 f_l(s) ds.
\end{equation}
Substituting the explicit formula \eqref{def-vm} $v_1 = V_0 (R)= a\cosh {R}$ into this equality we have
\begin{equation}\label{relation-R-alpha}
a^2 \cosh^2 (R) - a^2 = 2\int_{a \cosh (R)}^1 f_l (s) ds.
\end{equation}
Using the cubic bistable nonlinearity $f_l$ we can obtain
\begin{equation}\label{d-coshR-da}
a^2\frac{d [\cosh (R)] }{da} = \left. \frac{-2\int^1_{v_1 }f_l (s)ds- v_1  f_l (v_1 )}{ v_1 +f_l (v_1)} \right|_{v_1 = a\cosh (R)} <0.
\end{equation}
In fact, with
$$
F(v_1, \alpha):= 2\int^1_{v_1 }f_l (s)ds+ v_1  f_l (v_1)= -\frac12 v_1^4+ \frac13 (1+\alpha) v_1^3 +\frac{1-2\alpha}{6},
$$
we have $\frac{\partial F}{\partial \alpha}=\frac13 v_1^3- \frac13< 0$ and $F(v_1, \frac12)= -\frac12 v_1^4+ \frac12 v_1^3> 0$. Then \eqref{d-coshR-da} holds. This implies that
$$
\frac{dR}{da}<0.
$$

(2). In a similar way one can show that
$
\frac{dr}{da}<0$.

(3). We now show that $\frac{d\ell }{da} = \frac{d(R-r)}{da} <0$. Substituting $f_l(u)= u (u-\alpha )(1-u)$ into \eqref{equ 2.5}, we see that the $v$-coordinate of $D_1$ is given by the root of
\begin{equation}\label{v1-equ-root}
-\frac12 v_1^4+ \frac{2(\alpha+1)}{3}v_1^3+(1-\alpha)v_1^2+\frac{2\alpha-1}{6}- a^2 =0.
\end{equation}
From the phase diagram we see that, for any $a\in (0,\theta(\alpha))$, this equation has a unique root $v_1(a;\alpha) \in (a,1)$. Similarly, the $v$-coordinate of $D_2$ is given by the unique root $v_2(a;\alpha)\in (a,\theta)$ of following equation
\begin{equation}\label{v2-equ-root}
-\frac12 v_2^4+ \frac{2(\alpha+1)}{3}v_2^3+(1-\alpha)v_2^2- a^2 =0,
\end{equation}
Noticing that $v_1 = a\cosh {R(a)}$ and $v_2 = a\cosh {r(a)}$, we can conclude that
\begin{equation}\label{l-express}
\ell(a)= R(a)-r(a)= \ln \frac{v_1(a;\alpha)+\sqrt{v_1^2(a;\alpha)-a^2}}{a}- \ln \frac{v_2(a;\alpha)+\sqrt{v_2^2(a;\alpha)-a^2}}{a}.
\end{equation}
A tedious but trivial calculation gives the formulas of $v_1,\ v_2$ and $\ell(a)$, which shows that $\frac{d\ell (a)}{da}<0$ for any given $\alpha\in (0,\frac12)$ and any $a\in (0,\theta(\alpha)]$. A numerical simulation result is illustrated in Figure 2.
\begin{figure}[H]
\centering
\includegraphics[scale=0.5]{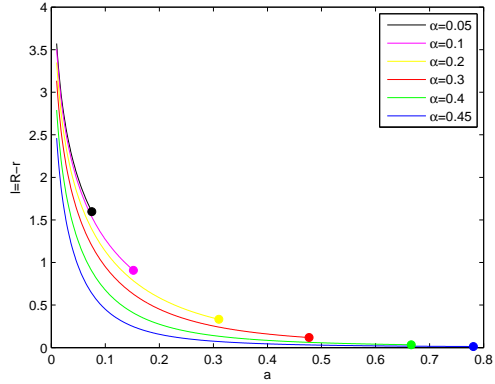}
\caption{For each given $\alpha\in (0,\frac12)$, $\ell(a)$ is strictly decreasing in $a\in (0,\theta(\alpha)]$. The $a$-coordinate of the end point of each curve is $\theta(\alpha)$.}
\end{figure}
\end{proof}

\subsection{Stationary solutions}

For any given $\alpha\in (0,\frac12)$, by the monotonicity of $\ell(a)$ on $a$, we see that
\begin{equation}\label{def-L*}
2 L^* := \ell(\theta(\alpha)) = \inf\limits_{0<a<\theta(\alpha)} \ell(a) >0.
\end{equation}
Clearly, this value depends only on $f(x,u)$. This is the critical value of $L$ appearing in our main theorems, and it plays a key role in the classification for the dynamics of the solutions.

Based on the monotonicity of $R,r,$ and $\ell$, we can construct some positive solutions of \eqref{v-eq} which will be used below.

\begin{prop}\label{prop:ss}
Let $L^*$ be the critical value defined by \eqref{def-L*}.

\begin{enumerate}[{\rm (a)}]
\item
In case $0<L<L^*$, the equation \eqref{v-eq} has exactly one positive solution $\mathcal{V}_b \in C^1(\R)\cap C^\infty (\R\backslash \{\pm L\})$ which satisfies \eqref{V-b-property} and
\begin{equation}\label{decay-to-1-rate}
0<1 - \mathcal{V}_b(x) \sim e^{-\sqrt{1-\alpha}\; |x|} \mbox{\ \ as\ \ }|x|\to \infty.
\end{equation}

\item In case $L=L^* $, the equation \eqref{v-eq} has at least two positive solutions, one is $\mathcal{V}_b$ as above, the other one $\mathcal{V}_s\in C^1(\R)\cap C^\infty (\R\backslash \{\pm L\})$ satisfies \eqref{V-s-property} and
\begin{equation}\label{V-s-property1}
0<1-\mathcal{V}_s(x) \sim e^{\sqrt{1-\alpha}\; x}\mbox{\ as\ }x\to -\infty,\quad 0<\mathcal{V}_s(x)\sim e^{-\sqrt{\alpha}\; x} \mbox{\ as\ }x\to \infty.
\end{equation}

\item In case $L>L^* $, the equation \eqref{v-eq} has at least three positive
solutions: $\mathcal{V}_b$, $\mathcal{V}_s$ as above, and $\mathcal{V}_g\in C^1(\R)\cap C^\infty (\R\backslash \{\pm L\})$ which satisfies, for some $x_1, x_2$ with $-L <x_1 <L <x_2$,
\begin{equation}\label{V-g-property1}
 \left\{
 \begin{array}{l}
  \mathcal{V}'_g(x)<0 \mbox{ for }x\in (-\infty, x_1)\cup (x_2, \infty),\quad \mathcal{V}'_g(x)>0 \mbox{ for }x\in (x_1, x_2),\\
   \mathcal{V}_g(x_1)\in (0,\theta),\quad \mathcal{V}_g(x_2)=\theta,\\
  0<1-\mathcal{V}_g(x) \sim e^{\sqrt{1-\alpha}\; x} \mbox{\ \ as\ \ }x\to -\infty,\quad 0< \mathcal{V}_g(x) \sim e^{-\sqrt{\alpha}\; x} \mbox{\ \ as\ \ }x\to \infty.
    \end{array}
  \right.
\end{equation}
 Moreover, \eqref{v-eq} has no other positive solutions satisfying $v(-\infty)=1$ and $v(\infty)=0$.

\item The solutions $\mathcal{V}_b, \mathcal{V}_s, \mathcal{V}_g$ are well-ordered, if they exist,
\begin{equation}\label{well-ordered}
\mathcal{V}_s(x) < \mathcal{V}_g(x) <\mathcal{V}_b(x),\quad x\in \R.
\end{equation}
\end{enumerate}
\end{prop}

\begin{figure}[htbp]
\centering
\includegraphics[scale=0.45]{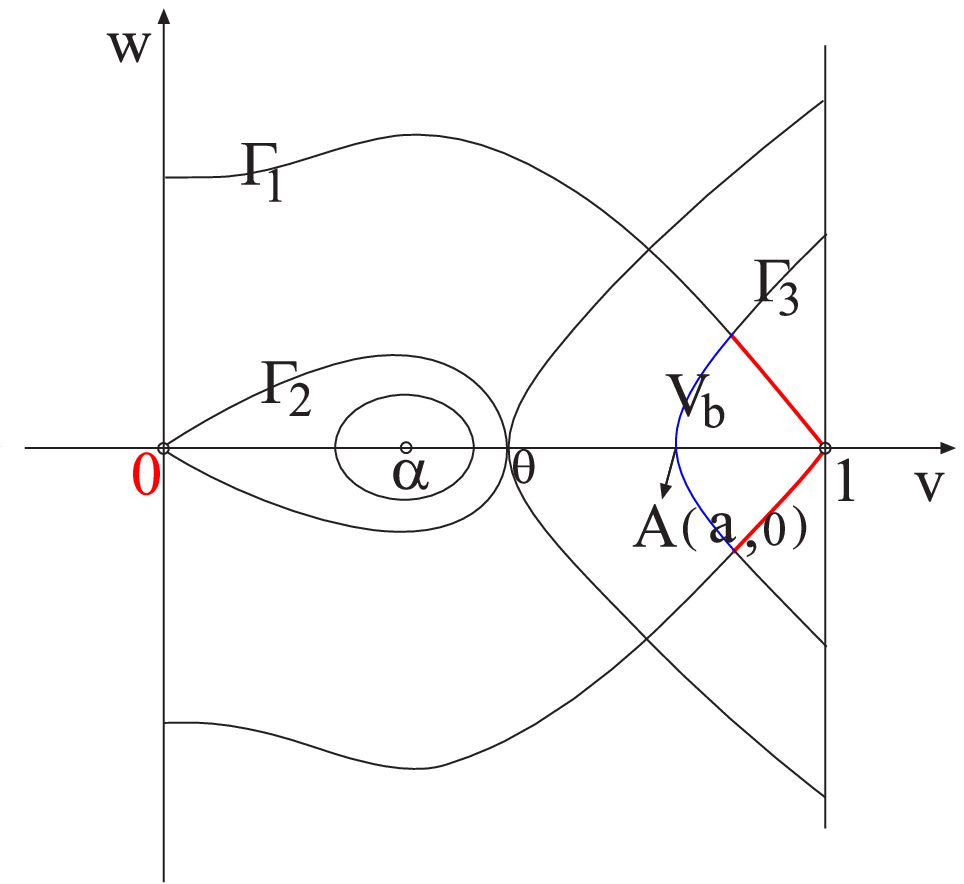}
\includegraphics[scale=0.45]{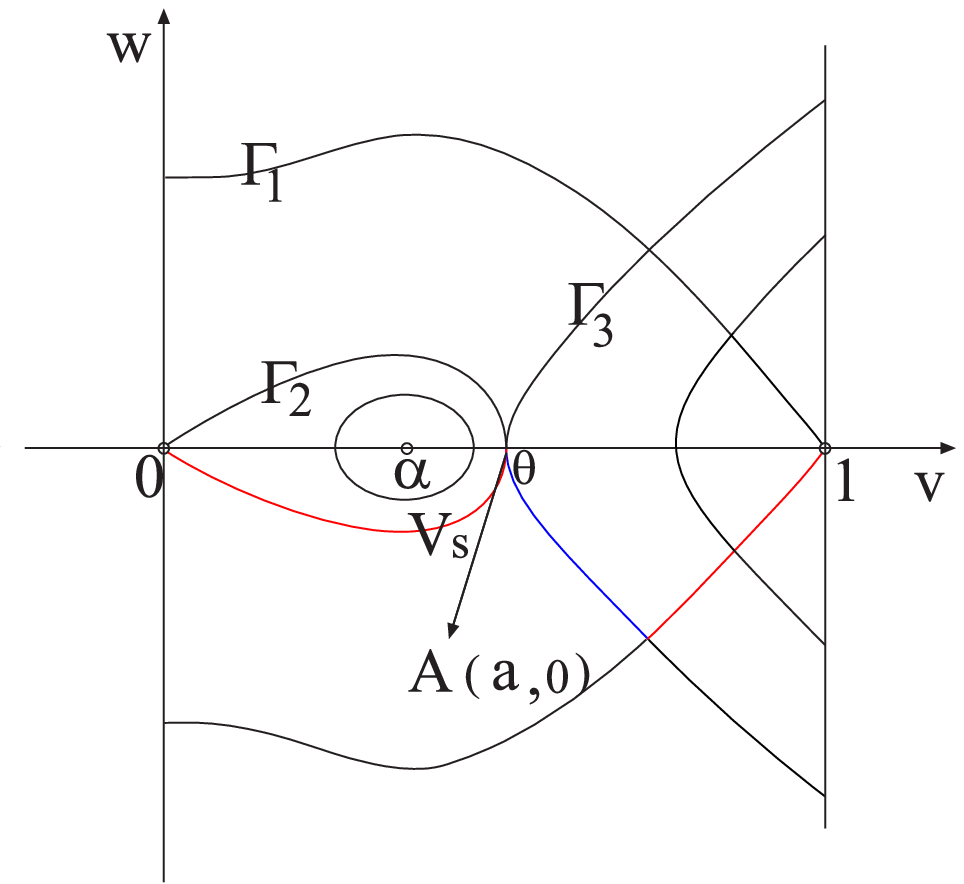}
\caption{\small
Trajectories of the solutions of \eqref{v-eq}: the case $L =L^* $. (a). Trajectory of $\mathcal{V}_b$; (b) Trajectory of $\mathcal{V}_s$.}
\end{figure}

\begin{figure}[H]
\centering
\includegraphics[scale=0.45]{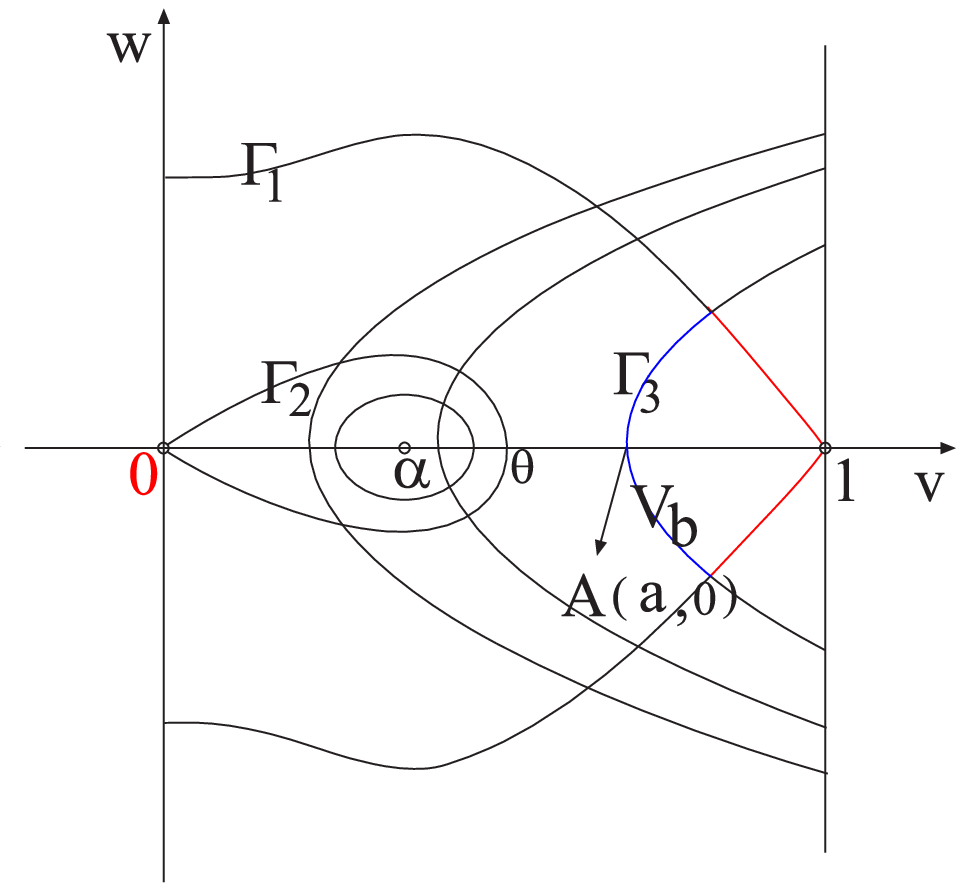}
\includegraphics[scale=0.45]{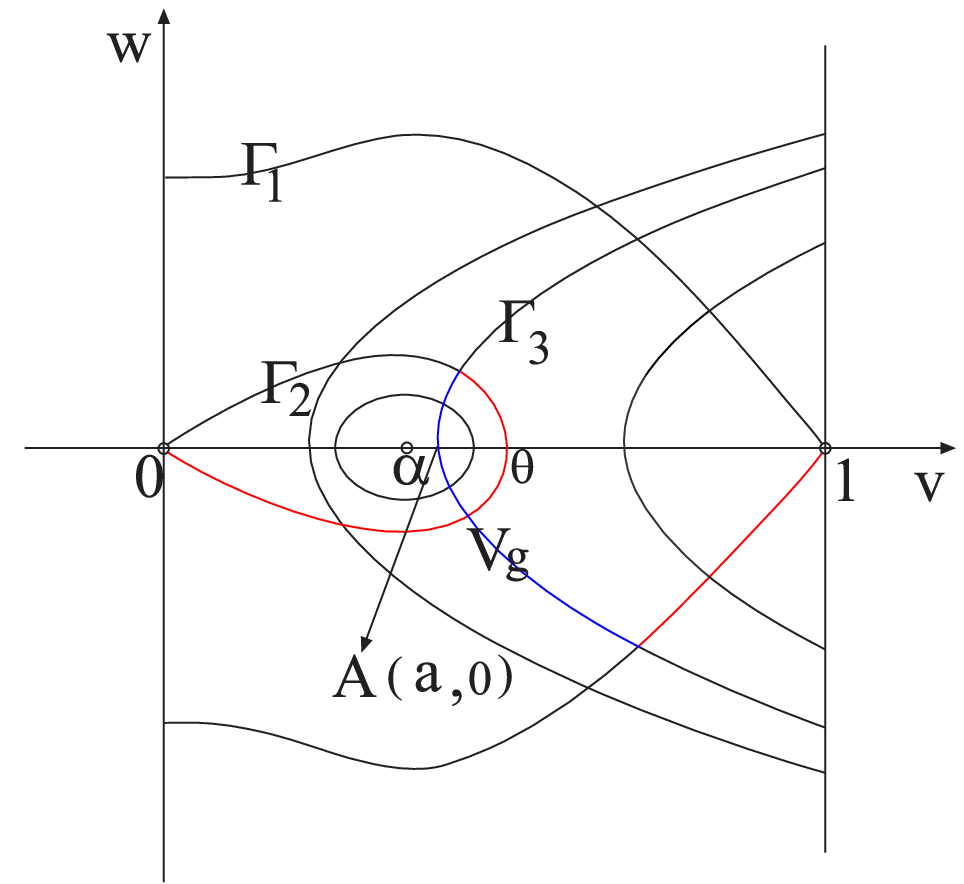}
\includegraphics[scale=0.45]{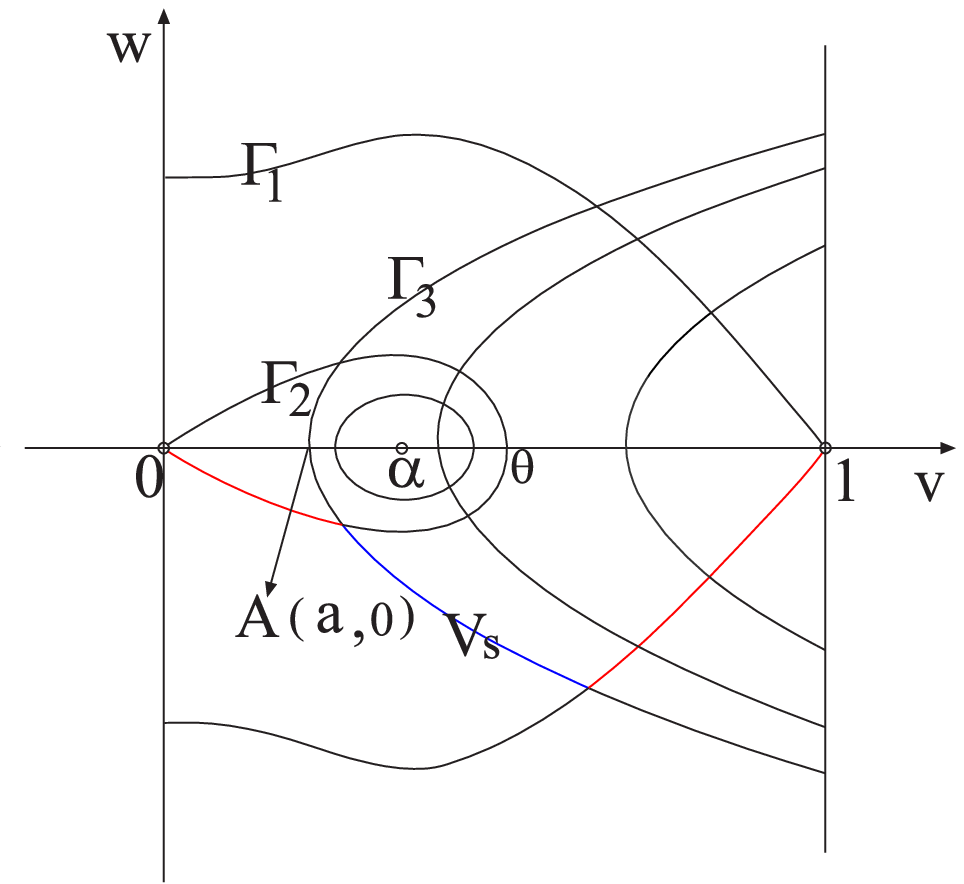}
\caption{\small Trajectories of the solutions of \eqref{v-eq}: the case $L >L^*$. (a) Trajectory of $\mathcal{V}_b$; (b) Trajectory of $\mathcal{V}_g$; (c) Trajectory of $\mathcal{V}_s$.}
\end{figure}

\begin{proof}
First, we construct solution $\mathcal{V}_b$ for any $L>0$. In Figure 3 (a), we see that the combination of the trajectory pieces: $(1,0)$-$B$-$A$-$C$-$(1,0)$ gives such a solution, as long as $X_{BAC}$ (which denotes the life span of the solution $V_0(x)$ given by the piece $BAC$) is $2L$.
From above we know that $X_{BAC}= 2R(a)$. Since
$$
R(a)\to 0\mbox{ as } a\to 1,\quad R(a)\to \infty \mbox{ as }a\to 0,
$$
and since $R(a)$ is strictly decreasing in $a$, there exists a unique $a\in (0,1)$ such that $2R(a)=2L$. For this choice of of $a$, the trajectory piece combination gives the solution $\mathcal{V}_b$.

Next, we construct $\mathcal{V}_s$ in case $L \geq L^*$. It is easily seen that the combination of trajectory pieces: $(1,0)$-$B$-$C$-$(0,0)$ in Figure 3 (b) or Figure 4 (c) gives a positive solution as $\mathcal{V}_s$, as long as $X_{BC}=2L$.  In case $L=L^*$, the unique suitable choice of $C$ is $(\theta,0)$ (see Figure 3 (b)). In case $L>L^*$, however, $X_{BC} = R(a)-r(a)=\ell(a)$, which is also strictly decreasing in $a$ with
$$
\ell(a) \to \infty \mbox{ as }a\to 0,\quad \ell(a)\to 2L^* \mbox{ as }a\to \theta.
$$
Hence, there exists a unique $a$ such that $\ell(a)=2L$. The corresponding combination of trajectory pieces defines the solution $\mathcal{V}_s$.

We then construct $\mathcal{V}_g$ in case $L>L^*$. As in Figure 4 (b), the combination of trajectory pieces: $(1,0)$-$B$-$C'$-$A$-$C$-$(\theta,0)$-$(0,0)$ gives a solution of type $\mathcal{V}_g$. Note that, for this purpose, we require that $X_{BAC} = R(a)+r(a)=2L$. From the analysis in the previous subsection this span is strictly decreasing in $a$. Hence there is a unique $a\in (0,\theta)$ satisfies this requirement. This prove the existence and uniqueness of $\mathcal{V}_g$.

The regularities of these solutions follow from the fact that they are bounded in the $W^2_{p, loc}(\R)$ topology, and smooth in the interior of $\R\backslash \{\pm L\}$ by the interior estimates in the theory of elliptic equations.
The decay rates in \eqref{decay-to-1-rate}, \eqref{V-s-property1}, \eqref{V-g-property1} can be shown by the corresponding solutions given by the trajectory pieces.

Finally, we prove the ordering \eqref{well-ordered}.  We only consider the case $L>L^*$, since the proof for $\mathcal{V}_b >\mathcal{V}_s$ in case $L=L^*$ is similar.
In Figure 4 (a), we write the coordinates of the points $A,B,C$ as $(a_1, 0),\ (v^B_1, w^B_1),\ (v^C_1, w^C_1)$, respectively, and write the coordinates of $A,B,C$ in Figure 4 (b) and (c) similarly with subscripts $1$ replaced by $2, 3$, respectively.

We first prove $\mathcal{V}_g>\mathcal{V}_s$. Recall that, the trajectory piece $BC$ in Figure 4 (c) gives a solution $V_0(x;a_3) = a_3 \cosh x$, its shift $a_3 \cosh (x - \frac{r(a_3)+R(a_3)}{2})$  coincides with $\mathcal{V}_s (x)$ in $[-L,L]$, and $\ell(a_3)=R(a_3) - r(a_3) = 2L$.
Similarly, the trajectory piece $BC'AC$ in Figure 4 (b) gives the solution $V_0(x;a_2)= a_2 \cosh x$, its shift $a_2 \cosh (x - \frac{\ell(a_2)}{2})$ coincides with $\mathcal{V}_g(x)$ in $[-L,L]$, and $\ell(a_2) + 2 r(a_2) =2L$. Hence,
$$
\ell(a_2) = 2L - 2r(a_2) < 2L =\ell (a_3).
$$
By the monotonicity of $\ell$ we see that
$$
a_3 < a_2.
$$
Thus, the trajectory passing through $(a_3,0)$ lies on the left of that passing through $(a_2,0)$. This implies that $(v^B_3, w^B_3)$ lies on the left of $(v^B_2, w^B_2)$,
$(v^C_3, w^C_3)$ lies on the left of $(v^{C'}_2, w^{C'}_2)$. Hence,
$$
v^B_3 <v^B_2,\quad v^C_3 = v^{C'}_3 < v^{C'}_2 = v^C_2,
$$
that is,
\begin{equation}\label{v3<v2-L}
 a_3 \cosh \left(\pm L - \frac{r(a_3)+R(a_3)}{2} \right) < a_2 \cosh  \left(\pm L - \frac{\ell (a_2) }{2} \right).
\end{equation}
We now prove
\begin{equation}\label{v3<v2-whole}
 a_3 \cosh \left( x - \frac{r(a_3)+R(a_3)}{2} \right) <  a_2 \cosh \left( x - \frac{\ell (a_2)}{2}\right),\quad x\in [-L,L].
\end{equation}
By contradiction we assume this is not always true. Then combining with \eqref{v3<v2-L} we see that, there exists $y \in (-L,L)$ such that
$$
 a_3 \cosh \left( y - \frac{r(a_3)+R(a_3)}{2} \right) =  a_2 \cosh \left( y - \frac{\ell (a_2)}{2}\right),
$$
and the derivatives of them at this point $y$ satisfy
$$
 a_3 \sinh \left( y - \frac{r(a_3)+R(a_3)}{2} \right)  \leq  a_2 \sinh \left( y - \frac{\ell (a_2)}{2}\right).
$$
This is impossible due to $a_3<a_2$. This proves $\mathcal{V}_s<\mathcal{V}_g$ in $[-L,L]$.

On the interval $(-\infty,-L]$, we have
$$
\mathcal{V}_s(x) = \tilde{V}_1 (x-x_3),\quad \mathcal{V}_g(x) = \tilde{V}_1 (x-x_2)
$$
for some suitable $x_3, x_2$, where $\tilde{V}_1(x)$ is the solution given by $\Gamma'_1$. By $v^B_3 < v^B_2$ we have
$$
 \tilde{V}_1 (-L-x_3)=\mathcal{V}_s(-L) = v^B_3 < v^B_2 = \mathcal{V}_g(-L) = \tilde{V}_1 (-L-x_2) .
$$
Hence $x_3 < x_2$ and so
$$
\mathcal{V}_s(x) = \tilde{V}_1 (x-x_3) <  \mathcal{V}_g(x) = \tilde{V}_1 (x-x_2),\quad x\leq -L.
$$
In a similar way one can show that $\mathcal{V}_s <\mathcal{V}_g$ in $[L,\infty)$. In summary we obtain the the first inequality in \eqref{well-ordered}.

The second inequality of \eqref{well-ordered} is proved similarly.
\end{proof}

\begin{remark}\label{rem:general-bi-ss1}
\rm
We constructed three positive stationary solutions $\mathcal{V}_b, \mathcal{V}_g, \mathcal{V}_s$ in the previous proposition. For convenience, we call them the {\bf big, transition and small stationary solutions}, respectively, and denote
\begin{equation}\label{def-omega-limit-set}
\mathcal{S} := \{\mathcal{V}_s, \mathcal{V}_g, \mathcal{V}_b\}.
\end{equation}
In the next section we will show that any $\omega$-limit of the solution of \eqref{main-eq} with initial data satisfying \eqref{ini-cond-2} is an element in $\mathcal{S}$, even though \eqref{main-eq} has other positive stationary solutions. For example, the reflections of the solutions in $\mathcal{S}$ with respect to $x=0$, and many  other types of solutions, positive in the whole $\R$, or compactly supported.
\end{remark}

Besides the positive stationary solutions mentioned above, we can construct some other upper and lower solutions which can be used for comparison. Here we give one of them.
In Figure 4 (b), we see that the combination of the trajectory pieces $F$-$D$-$B^*$-$E$-$\Gamma_4$ defines a function $V(x)$. More precisely, denote $D= (v^D, w^D),\ E=(v^E, w^E)$. Recall that we denote the solution corresponding to trajectory $\Gamma^*_5$ by $V^*_5(x)$ in the previous subsection. Denote the solution corresponding to the combination $D$-$B^*$-$E$ by $V^*_0(x)$, and the periodic solution corresponding to the $\Gamma_4$ by $V^{per}(x)$. Combining them together we obtain the function $V$:
$$
  V(x) =
  \left\{
  \begin{array}{ll}
  V^*_5(x+x_1), & x\leq -L,\\
  V^*_0(x+x_2), & x\in [-L,L],\\
  V^{per}(x+x_3), & x\geq L,
  \end{array}
  \right.
$$
where the shifts $x_1,x_2,x_3$ are chosen to satisfy the following matching conditions:
$$
 V^*_5 (-L+x_1)=v^D = V^*_0(-L+x_2),\quad  V^*_0(L+x_2)= v^E =V^{per} (L+x_3).
$$
Moreover, for any small $\varepsilon_0>0$, there exists $-\overline{L}< -L$ such that
$V(x)= V^*_5 (x + x_1) \geq 1+\varepsilon_0$ for $x\leq -\overline{L}$. Now we define
\begin{equation}\label{def-overline-V}
\overline{V}(x) :=
 \left\{
 \begin{array}{ll}
 1+ \varepsilon_0, & x\leq -\overline{L},\\
 V^*_5(x+x_1), &  x\in [-\overline{L}, -L],\\
 V^*_0(x+x_2), & x\in [-L, L],\\
 V^{per}(x+x_3), & x\geq L.
 \end{array}
 \right.
\end{equation}
Note that this function is continuous in $\R$ and $C^1$ in $\R\backslash \{-\overline{L}\}$. In addition,
$$
\overline{V}'(-\overline{L}-0)=0 > \overline{V}'(-\overline{L}+0),
$$
and
$$
\overline{V}(x) \geq \varepsilon_1 := \min\left\{ V^*_0 (0),\ \min\limits_{x\in \R} V^{per}(x)\right\} >0,\quad x\in \R.
$$
In the Subsection \ref{subsec:1.3} we will use this upper solution to separate $\mathcal{V}_s$ from $\mathcal{V}_g$ to give the trichotomy results.

\begin{remark}\label{rem:general-bi-ss2}\rm
As we mentioned in Section 1, one may be interested in general reaction terms, to say, $f$ is a bistable nonlinearity outside of the unfavorable region and a general negative growth rate in the region. In this case, the phase plane analysis we do in this section remains valid, but with more complicated situations. Now we give a sketch below. (1). We collect all possible trajectory pieces $BC$ as in Figure 3(c) and 4(c), and define
$$
L_* := \min \{X_{BC} \mid B\in \Gamma'_1,\ C \mbox{ lies on the lower half of } \Gamma_2\};
$$
then collect all possible trajectory pieces $BC'C$ as in Figure 4(b), and define
$$
L^* := \min \{X_{BC'C} \mid B \in \Gamma'_1,\ C \mbox{ lies on the upper half of } \Gamma_2\}.
$$
Then $L^*\geq L_* >0$. (Note that, in our current special case $L_* = L^* = X_{BC}$ for $C=(\theta,0)$.) (2). If we consider only positive stationary solutions of \eqref{v-eq} satisfying $v(-\infty)=1,\ v(\infty)=0$, then we can show the following results in a similar way as we do in this section: when $L<L_*$, \eqref{v-eq} has
only $\mathcal{V}_b$ type of solutions, which may be not unique but well-ordered; when $L\in [L_*, L^*)$ (or $L=L^*$ when $L_*=L^*$), \eqref{v-eq} has two groups of solutions, one group is of the form $\mathcal{V}_s$, the other is of the form $\mathcal{V}_b$. Each group may be not unique but well-ordered;
when $L\geq L^*$ (or $L>L^*$ when $L_*=L^*$), \eqref{v-eq} has three groups of solutions, they are of the form $\mathcal{V}_s, \mathcal{V}_g, \mathcal{V}_b$, respectively. Each group may be not unique but well-ordered. (3). Though we can derive the order in each group, it is not easy to prove all the solutions are well-ordered. Even, we can not easily to show,
by just analyzing the ODE \eqref{v-eq}, the existence of the smallest and largest stationary solution. However, it is remarkable to point out that, this last conclusion can be proved by a PDE approach.
More precisely, we consider an initial data $u_0(x)$ which approaches $1$ as $x\to -\infty$ in a sufficiently slow decay rate, and it is decreasing and takes $0$ at some point. If we shift this initial data leftward sufficiently far such that it is smaller than any stationary solution, then, by the general convergence result (which will be proved in the next section) the solution with this initial data converges to a positive stationary solution, which must be the smallest one.
On the other hand, the solution with initial data $u_0 \equiv 1$ must converges to the largest stationary solution.
\end{remark}

\section{General Convergence Result}

In this section we use the so-called zero number argument to prove that, for any solution $u$ of \eqref{main-eq}, its $\omega$-limit set is contained in $\mathcal{S}$. We first recall the typical zero number diminishing properties. Consider
\begin{equation}\label{linear}
\eta_t=a(x,t) \eta_{xx}+b(x,t) \eta_x+c(x,t) \eta\quad \mbox{ in } E_0 := \{(x,t) \mid x\in \R,\ t\in (t_1, t_2) \},
\end{equation}
where $t_2 > t_1\geq 0$. For each $t\in (t_1, t_2)$, denote by
$$
\mathcal{Z}(t) := \sharp \{x\in \R\mid \eta(\cdot ,t)=0\}
$$
the number of zeroes of $\eta (\cdot,t)$ in $\R$. A point $x_0$ is called a {\it multiple zero} (or {\it degenerate zero}) of $\eta (\cdot, t)$ if $\eta (x_0,t) = \eta_x (x_0,t)=0$. In 1988, Angenent \cite{A} proved a zero number diminishing property, and in 1998, the conditions Angenent had used in \cite{A} were weaken by Chen \cite{ChXY} for strong solutions in $W^{2,1}_{p,\; loc}(\R\times (0,\infty))$. One of their results is summarized as the following:

\begin{lem}[\cite{A, ChXY}]\label{lem:zero}
Assume the coefficients in \eqref{linear} satisfies
\begin{equation}\label{smoothy}
a, a^{-1}, a_t, a_x, b, c\in L^\infty.
\end{equation}
Let $\eta$ be a nontrivial $W^{2,1}_{p,loc}$ solution of \eqref{linear}. Then
\begin{itemize}
\item[(1)] the zeros of $\eta(\cdot,t)$ are isolated;
\item[(2)] if $\mathcal{Z}(t)<\infty$ for some $t_0\in (t_1, t_2)$, then it is decreasing in $t\in (t_0, t_2)$. Moreover, if $s\in (t_0, t_2)$ and $x_0$ is a multiple zero of $\eta(\cdot,s)$, then $\mathcal{Z} (s_1) > \mathcal{Z} (s_2)$ for all $s_1, s_2$ satisfying $t_0 < s_1< s< s_2 <t_2$.
\end{itemize}
\end{lem}

Using this lemma we can prove the following convergence result.

\begin{thm}\label{thm:conv}
Assume {\rm (H)}, $u_0$ satisfies   \eqref{ini-cond-2}. Then for any $L>0$, the unique time-global solution $u(x,t)$ converges as $t\to \infty$ to an element in $\mathcal{S}$, in the topology of $L^\infty_{loc}(\R)$.
\end{thm}

\begin{proof}
First, one can derive the boundedness of $u$ easily by the maximum principle. Then the existence and uniqueness of the strong solution in $W^{2,1}_{p, loc}(\R\times (0,\infty))$ follows from the standard parabolic theory.
By the standard $L^p$ estimates, for any increasing time sequence $\{t_n\}$, any $p>1,\ M>0$, there exists $C=C(M,p)$ such that
$$
\|u(x,t +t_n)\|_{W^{2,1}_p ([-M,M]^2)} \leq C,
$$
therefore, for any $\nu\in \left(0,1-\frac1p \right)$, there exists a subsequence of $\{t_n\}$, denoted it again by $\{t_n\}$, and a function $w(x,t)\in C^{1+\nu, \frac{1+\nu}{2}}([-M,M]^2)$ such that
$$
\| u(x,t+t_n)- w(x,t)\|_{C^{1+\nu, \frac{1+\nu}{2}}([-M,M]^2)} \to 0 \mbox{\ \ as\ \ }n\to \infty.
$$
By Cantor's diagonal argument, the $\omega$-limit set of $u(\cdot,t)$ in the topology of $C^{1+\nu}_{loc}(\R)$ is not empty, and by standard dynamics theory, compact, invariant and connected.

We divide the rest proof into several steps.

\medskip
{\it Step 1. A quasi-convergence result: any $\omega$-limit of $u$ is a stationary solution}. We follow the idea in \cite{DL,DM} and use the zero number argument.
We show that, for any $t\in \R$, $w(\cdot,t)$ is actually a stationary solution. We only need to prove $w(\cdot,0)$ is so. By contradiction, we assume that $w(x,0)$ is not a stationary solution, so $w(x,0)\not\equiv 0$. Assume without loss of generality that $w(0,0)>0$. Using $w(x,0)$ we construct a real stationary solution:
\begin{equation}\label{construct-ss}
 \left\{
  \begin{array}{l}
  v''+f(x,v)=0,\ \quad x\in \R,\\
  v(0)= w(0,0)>0,\ \ v'(0)=w_x(0,0).
  \end{array}
  \right.
\end{equation}
By the phase plane analysis in the previous section we see that, the bounded $W^2_p$ solution of \eqref{construct-ss} must be less than $1$ and positive in some open interval $(l_0,r_0)$ with $-l_0, r_0\in (0,\infty]$. In addition, it has the following cases:
\begin{enumerate}[(1).]
\item $v=v_1$ satisfies $v_1(l_0)=v_2(r_0) =0$ for $l_0<0<r_0$;
\item $v=v_2$ satisfies $v_2(l_0)=0$ for $l_0<0$, $v_2(x)>0$ for $x>l_0$ and $v_2$ satisfies \eqref{limit-to-1} as $x\to \infty$;

\item $v=v_3$ satisfies $v_3(l_0)=0$ for $l_0<0$, $v_3(x)>0$ for $x>l_0$ and $v_3$ satisfies \eqref{limit-to-per} for $x>L$;

\item $v=v_4$ satisfies $v_4(l_0)=0$ for $l_0<0$, $v_4(x)>0$ for $x>l_0$ and $v_4$ satisfies \eqref{limit-to-0} as $x\to \infty$;

\item $v=v_5$ satisfies $v_5(r_0)=0$ for $r_0>0$, $v_5(x)>0$ for $x<r_0$ and $v_5$ satisfies \eqref{limit-to-1} as $x\to -\infty$;

\item $v=v_6$ satisfies $v_6(r_0)=0$ for $r_0>0$, $v_6(x)>0$ for $x<r_0$ and $v_6$ satisfies \eqref{limit-to-per} for $x<-L$;

\item $v=v_7$ satisfies $v_7(r_0)=0$ for $r_0>0$, $v_7(x)>0$ for $x<r_0$ and $v_7$ satisfies \eqref{limit-to-0} as $x\to -\infty$;

\item $v=v_8$ satisfies $v_8(x)>0$ for all $x\in \R$. Moreover, it satisfies \eqref{limit-to-1} as $x\to -\infty$; and satisfies \eqref{limit-to-1} or \eqref{limit-to-0} as $x\to \infty$. Note that such solutions ar nothing but $\mathcal{V}_b, \mathcal{V}_g, \mathcal{V}_s$ constructed in the previous section, and each of them is unique if it exists;

\item $v=v_9$ satisfies $v_9(x)>0$ for all $x\in \R$. Moreover, it satisfies \eqref{limit-to-1} as $x\to -\infty$; and satisfies \eqref{limit-to-per} for $x>L$.
    (Note that $V^{per}$ we constructed in the previous section belongs to this type, and only exist in case $L>L^*$);

\item $v=v_{10}$ satisfies $v_{10}(x)>0$ for all $x\in \R$. Moreover, it satisfies  \eqref{limit-to-per} for $x<-L$, or satisfies \eqref{limit-to-0} as $x\to -\infty$;
    and satisfies one of the following decay rates on the right side:
\end{enumerate}
\begin{equation}\label{limit-to-1}
0<1-v(x) \sim e^{-\sqrt{1-\alpha}|x|};
\end{equation}
\begin{equation}\label{limit-to-per}
v(x)=V^{per}(x) \mbox{ for some positive periodic solution } V^{per} \mbox{ with } 0<V^{per}(x)<\theta;
\end{equation}
\begin{equation}\label{limit-to-0}
0<v(x) \sim e^{-\sqrt{\alpha}|x|}.
\end{equation}

For any of the above stationary solution $v_i$, we set $\eta(x,t):= u(x,t)-v_i(x)$ for $x\in (l_0,r_0),\ t\geq 0$. Then $\eta$ satisfies
$$
\eta_t = \eta_{xx} +c(x,t)\eta,\quad (l_0,r_0),\ t>0,
$$
where
$$
c(x,t):= \left\{
 \begin{array}{ll}
 \displaystyle \frac{f(x,u)-f(x,v_i)}{\eta(x,t)},& \mbox{when } \eta(x,t)\not=0,\\
 0, & \mbox{when } \eta(x,t)=0
 \end{array}
 \right.
$$
is bounded. Hence, the zero number diminishing property Lemma \ref{lem:zero} is applicable. On the other hand, the decay rates to $1$ and/or to $0$ for $u_0$ in \eqref{ini-cond-2} is imposed in order to separate $u_0$ with any one of the stationary solutions: for any choice of $i\in \{1,2,\cdots,10\}$, the function $\eta$ satisfies
$$
\eta(l_0,t_1)\not= 0, \quad \eta(r_0,t_1)\not=0,
$$
for any sufficiently small $t_1>0$. Moreover, when $l_0=-\infty$ we have $\eta(x,t_1)\not= 0$ for all $x\ll -1$, and
when $r_0=\infty$ we have $\eta(x,t_1)\not= 0$ for all $x\gg 1$. Therefore, the
zero number property implies that the number of zeros, denoted by $\mathcal{Z}(t)$, of $\eta(\cdot,t)$ in $(l_0,r_0)$ is finite for all $t>t_1$: $\mathcal{Z}(t)<\infty$.
In addition, it is decreasing in $t$ and strictly decreasing in $t$ when it passes a moment when there is degenerate zeros. Consequently, after some time, to say, when $t\geq T$ for some large $T$, $\eta(\cdot,t)$ not longer has degenerate zeros. As it was shown in \cite[Lemma 2.6]{DM}, this implies that any $\omega$-limit of $\eta(\cdot,t)$ either is identical $0$, or has only simple zeros. Especially, the limit $w(x,0)-v_i(x)$ of $\eta(x,t_n)$ is identical zero, this is what we desired, or it has only simple zeros. The latter, however, contradicts the initial conditions in \eqref{construct-ss}.
This proves the quasi-convergence in Step 1.

\medskip
{\it Step 2. To show $\omega(u)\subset \mathcal{S}$.}
By our assumption \eqref{ini-cond-2} we see that for any $a\in (\theta,1)$ and any $x_1\ll -1$,
$$
u_0(x) > V_a (x-x_1),\quad x\in [x_1 -x_0, x_1 +x_0],
$$
for the compactly supported stationary solution $V_a (x)$ given in \eqref{ss-bistable}. By comparison we have $u(x,t)\geq V_a (x-x_1)$ for all $t>0$. So the possible choices of the $\omega$-limits of $u$ are as $v_5, v_6, v_7, v_8,v_9$. On the other hand, $u$ is bounded in the topology of $W^{2,1}_{p, loc}(\R\times (0,\infty))$ and so the $\omega$-limits can be taken in the topology of $C^{1+\nu, \frac{1+\nu}{2}}_{loc}(\R\times (0,\infty))$ for any $\nu \in (0,1)$. This implies that any $\omega$-limit must be a $C^{1+\nu}(\R)$ function. Hence, $v_5,v_6$ and $v_7$ are excluded from the candidates of the $\omega$-limits since they are not $C^1$ function at $x=r_0$.

We finally exclude the possibility of $v_9$. By contradiction, assume $v_9(x)$ is an $\omega$-limit of $u$. As we mentioned in the previous section, once a solution of $v_9$ type exists, there must be a small family of such type of solutions. We choose one of them, to say, $\tilde{v}_9$ such that
$$
\max v_9 \not= \max \tilde{v}_9,\quad \min v_9 \not= \min \tilde{v}_9.
$$
Now we consider the number of zeros $\mathcal{Z}_9 (t)$ of $\eta_9 (x,t):= u(x,t)- \tilde{v}_9 (x)$. On the one hand, as before we know by the assumption for $u_0$ that $\mathcal{Z}_9(t)<\infty$ and is decreasing in $t$. On the other hand, the assumption $v_9$ is an $\omega$-limit of $u$
implies that $\mathcal{Z}_9(t)$ tends to the number of the zeros of $v_9(x)-\tilde{v}_9(x)$, which is infinite, a contradiction.
Therefore, any $\omega$-limit of $u$ must be a stationary solution of type $v_8$, which are nothing but one of $\mathcal{V}_b, \mathcal{V}_g, \mathcal{V}_s$ in $\omega(u)\subset \mathcal{S}$.
This proves the conclusion in Step 2.

\medskip
{\it Step 3. To show the convergence.}  Since $\omega(u)\subset \mathcal{S}$ and $\mathcal{S}$ has only three isolated element, we conclude that each solution $u$ of \eqref{main-eq} with initial data satisfying   \eqref{ini-cond-2}
must converges to $\mathcal{V}_b, \mathcal{V}_g$ or $\mathcal{V}_s$.

This proves Theorem \ref{thm:conv}.
\end{proof}

\begin{remark}\label{rem:weaken-ini-cond}
\rm
From the proof of the above theorem we see that the key point to use the zero number argument is to guarantee the number of the intersection points between the initial data $u_0$ and any one of $v_i$, denoted it by $\mathcal{Z}[u_0 -v_i]$, is finite. We impose decay rates for $u_0$ in \eqref{ini-cond-2} to ensure this is true.
We remark that this is the only place to use the decay rate assumption.
Clearly, this condition can be extended. For example, if the initial data $u_0$ satisfies
$$
h:= \lim\limits_{x\to -\infty} u_0(x) \in (\alpha, 1),
$$
and satisfies similar conditions on the right hand side, then the zero number argument is applicable as before. In fact, in case $h\in [\theta,1)$, the finiteness of $\mathcal{Z}[u_0 -v_i]$ is obvious; in case $h\in (\alpha,\theta)$, then one can show that $u(x,t)>\theta$ for all large $t$ and $x\ll -1$ by a similar argument as in \cite{FP,FM}, and so it has at most finite number of intersection points with any stationary solution.
\end{remark}

\section{Asymptotic Behavior of the Solutions}
In this section we consider the asymptotic behavior for the solutions based on the general convergence result in Theorem \ref{thm:conv}, and prove our main theorems.

\subsection{Convergence result for small $L$}
\

\noindent
{\it Proof of Theorem \ref{thm1.1}}. When $L<L^*$, we have only one stationary solution $\mathcal{V}_b$ in $\mathcal{S}$. Hence, by the general result in Theorem \ref{thm:conv} we see that spreading happens for all $u$.
\hfill $\Box$

\subsection{Dichotomy result for critical $L^*$}

First, we give a sufficient condition for spreading.

\begin{lem}\label{lem3.5}
Assume \eqref{condition1}, $u_0$ satisfies   \eqref{ini-cond-2}. Assume further that, for some $a\in (\theta,1)$,
\begin{equation}\label{suff-cond-spreading}
u_0(x) \geq V_a (x-x_1),\quad  x\in [x_1 -x_0, x_1+x_0],
\end{equation}
where $V_a $ is the compactly supported stationary solution in \eqref{ss-bistable} with support $[-x_0,x_0]$, $x_1> x_0 +L$.
Then spreading happens for $u$.
\end{lem}

\begin{proof}
Sine $V_a (x-x_1)$ is a stationary solution, by comparison we have
$$
u(x,t) > V_a (x-x_1),\quad  x\in [x_1 -x_0, x_1+x_0],\ t>0.
$$
Hence the $\omega$-limit of $u$ is also larger than $V_a (x-x_1)$. By the previous theorem, suitable choice for such $\omega$-limit must be $\mathcal{V}_b$.
\end{proof}

\noindent
{\it Proof of Theorem \ref{thm1.2}}. Since $\mathcal{S}=\{\mathcal{V}_b, \mathcal{V}_s\}$ in the current case, we have either spreading or residue happens.

The prove the theorem we first show that spreading happens for large $\sigma$.
In fact, by $\phi(-\infty)=1$ in \eqref{ini-cond-2} we see that, when $\sigma\gg 1$ we have
$$
\phi(x-\sigma) \geq V_a (x- x_0 -L),\quad x\in [L, L+2x_0],
$$
for some stationary solution $V_a $ constructed in \eqref{ss-bistable}.
Therefore spreading happens for $u_\sigma$ by Lemma \ref{lem3.5}.

Denote
$$
\Sigma_1 := \{\sigma \in \R \mid \mbox{spreading happens for } u_\sigma(x,t)\}.
$$
Then $\Sigma_1$ is nonempty. By the assumption that $\phi(x)$ is decreasing function, so $\phi(x-\sigma)$ and $u_\sigma$ are increasing in $\sigma$. Hence $\Sigma_1$ is an interval. Moreover, for any $\sigma_1\in \Sigma_1$, $u_{\sigma_1}$ converges as $t\to \infty$ in the $L^\infty_{loc}(\R)$ topology to $\mathcal{V}_b$. Hence, there exists $T_1$ large such that
$$
u_{\sigma_1} (x,T_1) > V_a(x- x_0 -L),\quad x\in [L, L+2x_0].
$$
By the continuously dependence for the solution $u_\sigma$ on its initial data $\phi(x-\sigma)$, we see that, for any $\sigma$ satisfying $0<\sigma_1 -\sigma \ll 1$, we have
$$
u_{\sigma} (x,T_1) \geq V_a(x- x_0 -L),\quad x\in [L, L+2x_0].
$$
Therefore, spreading also happens for $u_{\sigma}$. This proves that $\Sigma_1$ is an open interval $(\sigma^*,\infty)$ for some $\sigma^*\in [-\infty, \infty)$.

If $\sigma^* = -\infty$, then there is nothing left to prove. If $\sigma^*\in \R$, then, for any $\sigma>\sigma^*$, spreading happens for $u_{\sigma}$; for any $\sigma\in (-\infty,\sigma^*]$, the only $\omega$-limit of $u_\sigma$ is $\mathcal{V}_s$, that is, residue happens.

This proves Theorem \ref{thm1.2}. \hfill $\Box$

\subsection{Trichotomy result for large $L$: Proof of Theorem \ref{thm1.3}}\label{subsec:1.3}.

By Proposition \ref{prop:ss} and Theorem \ref{thm:conv}, in case $L>L^*$, any solution of \eqref{main-eq} with initial data satisfying   \eqref{ini-cond-2} must converges to one element in $\mathcal{S}$.

Denote
\begin{equation*}
\begin{aligned}
\Sigma_1 := &\{\sigma\in \R\; |\; \mbox{spreading happens for }u_\sigma\};\\
\Sigma_2:= &\{\sigma \in \R\; |\; u_\sigma \mbox{ is a transition solution}\};\\
\Sigma_3:= &\{\sigma \in \R\; |\; \mbox{residue happens for } u_\sigma\}.
\end{aligned}
\end{equation*}
As in the proof of the previous theorem, one can show that $\Sigma_1$ is a nonempty open interval $(\sigma^*, \infty)$ for some $\sigma^*\in [-\infty, \infty)$.  In what follows we show $\Sigma_3$ is not empty and so $\sigma^* \in \R$. The proof is divided into several steps.

\medskip
{\it Step 1. To show $\Sigma_3$ is not empty}. We will use the the weak upper solution $\overline{V}$ of \eqref{main-eq} constructed at the end of Section 2. Recall that, for some $\varepsilon_0>0$,
\begin{equation}\label{left=1+epsilon}
\overline{V}\equiv 1+\varepsilon_0, \quad x\leq -\overline{L},
\end{equation}
and, for some $\varepsilon \in (0,\varepsilon_0]$,
\begin{equation}\label{bar-V>V-s}
\overline{V}(x) - \mathcal{V}_s(x) \geq 3 \varepsilon,\quad x\in \R.
\end{equation}

By our assumption \eqref{ini-cond-2} we see that, when $\sigma\ll -1$, the function $\phi(x-\sigma)< \overline{V}(x)$. Since $\overline{V}$ is a weak upper solution, by comparison, any $\omega$-limit of $u_\sigma(x,t)= u(x,t; \phi(x-\sigma))$ is less than $\overline{V}$, which is nothing but $\mathcal{V}_s$. This proves that $\Sigma_3$  is not empty.

\medskip
{\it Step 2. To show $\Sigma_3$ is an open set}.
Assume $\sigma_1\in \Sigma_3$, we now show that $\sigma_2 \in \Sigma_3$ when  $0<\sigma_2 -\sigma_1 \ll 1$. This implies that $\Sigma_3$ in open. For simplicity, denote the corresponding solutions of \eqref{main-eq} with initial data $\phi(x-\sigma_i)$  by $u_i$ ($i=1,2$). By the definition, $u_1(\cdot,t)$ converges as $t\to \infty$ to $\mathcal{V}_s$ in the $L^\infty_{loc}(\R)$ topology. Hence, there exists $T>0$ such that
$$
\| u_1 (x,t)- \mathcal{V}_s(x)\|_{L^\infty ([-\overline{L}, L])} \leq \varepsilon,\quad t\geq T.
$$
By parabolic estimate, for any $\nu\in (0,1)$, we actually have
$$
\| u_1 (x,t)- \mathcal{V}_s(x)\|_{C^{1+\nu} ([-\overline{L}, L])} \leq \varepsilon_1 ,\quad t\geq T,
$$
for some small $\varepsilon_1>0$. Taking $\varepsilon$ smaller if necessary we see that $\varepsilon_1$ is also small and so
\begin{equation}\label{u1x-L<0}
u_{1x}(L,t) < \mathcal{V}'_s (L) + \varepsilon_1 <0,\quad t\geq T.
\end{equation}
Moreover, by \eqref{left=1+epsilon} we have
\begin{equation}\label{small-for-x<L}
\| u_1 (x,t)- \mathcal{V}_s(x)\|_{L^\infty ((-\infty, L])} \leq \varepsilon,\quad t\geq T.
\end{equation}
Note that this inequality is not necessarily true in $J:= [L, \infty)$ since, till now, we have no monotonicity in this interval. Actually, a solution $u$ can really be not monotonically decreasing in this interval, especially for the spreading and transition solutions.
Nevertheless, we can claim that $u_1$ is monotonically decreasing for $x\gg 1$ and all large $t$. This is proved in two claims.

\medskip
\noindent
{\bf Claim 1}. There exists $M_1 =M_1 (T)>L+1$ such that $u_{1x}(x,T)<0$ for $x\geq M_1$.

\noindent
In fact, due to the monotonicity of $\phi(x-\sigma_1)$, by using the zero number argument in $J$ we see that any local maximum points of $u_1$ in $J$ must first arise at $x=L$ and then propagates rightward to the interior of $J$. Denote the right-most maximum point by $\xi(t)$, that is,
$$
u_{1x}(\xi(t),t)=0, \quad u_{1x}(x,t)<0 \mbox{ for } x>\xi(t),\ t>T_0,
$$
where $T_0$ is the first time when $\xi(t)$ appears in $J$. Since $u_{1x}$ satisfies a linear parabolic equation, by a simple estimate for linear equation we can derive that
$$
\xi(t)\leq M_1 (T),\quad t\in [T_0,T].
$$
This proves Claim 1.

\medskip
\noindent
{\bf Claim 2}. There exists $M_2>M_1$ such that $u_{1x}(x,t)<0$ for all $x\geq M_2, t\geq T$.

In fact, by Claim 1, there exists $M_2>M_1$, such that
\begin{equation}\label{Vs<epsilon-M2}
\mathcal{V}_s(x)<\varepsilon,\quad x\geq M_2,
\end{equation}
and
$$
u_1(x,T) > \tilde{u}_1(x,T), \quad x\in [L,M_2),
$$
where $\tilde{u}_1(x,t) := u_1(2M_2 -x,t)$ for $x\in [L,M_2],\ t\geq T$ denotes the reflection of $u_1$ with respect to $M_2$, and so $\eta(x,t):= u_1(x,t) - \tilde{u}_1(x,t)$ satisfies
$$
 \left\{
 \begin{array}{ll}
 \eta_t = \eta_{xx} + c(x,t) \eta,& x\in [L,M_2],\ t\geq T,\\
 \eta_x (L,t) <0, \quad \eta(M_2,t) =0, & t\geq T,\\
 \eta(x,T)>0, & x\in [L,M_2),
 \end{array}
 \right.
$$
for some bounded function $c$. By the maximum principle we have $\eta(x,t)>0$ for all
$x\in [L,M_2)$ and $t\geq T$.  If the right-most maximum $\xi(t)$ of $u_1$ moves to $M_2$ at some time $T_1>T$, then we have
$$
\eta(M_2, T_1)=0, \quad  \eta_x(M_2,T_1) = 2 u_{1x} (M_2, T_1) = 2u_{1x}(\xi(T_1), T_1)=0,
$$
this, however, contradicts the Hopf lemma. Thus $\xi(t)$ will never move rightward through
$M_2$.

\medskip
Now using the $L^\infty_{loc}(\R)$ convergence of $u_1$ to $\mathcal{V}_s$ we conclude that, for some large $T_2 >T_1 $,
$$
\| u_1 (x,t)- \mathcal{V}_s(x)\|_{L^\infty ((-\infty, M_2])} \leq \varepsilon,\quad t\geq T_2.
$$
Consequently, by Claim 2 and \eqref{Vs<epsilon-M2} we have
$$
u_1(x,t) \leq u_1(M_2, t) \leq \mathcal{V}_s (M_2)+\varepsilon \leq 2 \varepsilon ,\quad x\geq M_2,\ t\geq T_2.
$$
Together with \eqref{bar-V>V-s} we conclude that
$$
u_1(x,T_2) \leq \overline{V}(x) -\varepsilon,\quad x\in \R.
$$

By the continuous dependence of the solution $u_\sigma$ on its initial data $\phi(x-\sigma)$, we see that when $\sigma_2>\sigma_1$ with $\sigma_2 -\sigma_1 $ sufficiently small we have
$$
u_2(x,T_2) \leq \overline{V}(x),\quad x\in \R.
$$
This indicates that residue happens for $u_2$. This proves the openness of $\Sigma_3$ in Step 1. Furthermore, by comparison, there exists $-\infty < \sigma_* \leq \sigma^*$ such that $\Sigma_3 := (-\infty, \sigma_*)$.

\medskip
{\it Step 3. To complete the proof of Theorem \ref{thm1.3}}. From above we have
$$
\Sigma_1 = (\sigma^*,\infty),\quad \Sigma_3 =(-\infty, \sigma_*).
$$
Therefore, $\Sigma_2 =[\sigma_*, \sigma^*]$ is a nonempty closed set.

This proves Theorem \ref{thm1.3}.  \hfill $\Box$


\begin{thebibliography}{}
\bibitem{A}
S. B.~Angenent,
{\em The zero set of a solution of a parabolic equation}, J. reine angew. Math., 390 (1988), 79-96.

\bibitem{AW1}
D. G. Aronson and H. F. Weinberger,
{\em Nonlinear diffusion in population genetics, combustion, and nerve pulse propagation},
Lect. Notes Math., 446 (1975), 5-49.

\bibitem{AW2}
D. G. Aronson and H. F. Weinberger,
{\em Multidimensional nonlinear diffusion arising in population genetics},
Adv. Math., 30 (1978), 33-76.


\bibitem{BN}
H. Berestycki and J.-P. Nadal,
{\em Self-organised critical hot spots of criminal activity}, European J. Appl. Math., 21 (2010), 371-399.

\bibitem{BRR}
H. Berestycki, N. Rodríguez and L. Ryzhik,
{\em Traveling wave solutions in a reaction-diffusion model for criminal activity},
Multiscale Model. Simul., 11 (2013), 1097-1126.



\bibitem{ChXY} X. Y. Chen,
{\em A strong unique continuation theorem for parabolic equations}, Math. Ann.,  311 (1998), 603-630.

\bibitem{DL}
Y. Du and B. Lou,
{\em Spreading and vanishing in nonlinear diffusion problems with free boundaries},
J. Eur. Math. Soc., 17 (2015), 2673-2724.


\bibitem{DM}
Y. Du and H. Matano,
{\em Convergence and sharp thresholds for propagation in nonlinear diffusion problems}, J. Eur. Math. Soc., 12 (2010), 279-312.

\bibitem{FP}
E. Feireisl and P. Pol\'{a}\v{c}ik,
{\em Structure of periodic solutions and asymptotic behavior for time-periodic reaction–diffusion equations on $\R$}, Adv. Differential Equations, 5 (2000), 583-622.

\bibitem{FM}
P. C. Fife and J. B. McLeod,
{\em The approach of solutions of nonlinear diffusion equations to travelling front solutions}, Arch. Ration. Mech. Anal., 65 (1977), 335-361.

\bibitem{LWZ}
G. Lan, C. Wei and S. Zhang,
{\em Long time behaviors of single-species population models with psychological effect and impulsive toxicant in polluted environments}, Phys. A, 521 (2019), 828-842.

\bibitem{LZC}
B. Li, M. Zhang and B. Coffman,
{\em Can a barrier zone stop invasion of a population?}, J. Math. Biol., 81 (2020), 1193-1216.

\bibitem{ML}
G. A. Maciel and F. Lutscher,
{\em Allee effects and population spread in patchy landscapes}, J. Biol. Dyn., 9 (2015), 109-123.

\bibitem{Wang1}
M. Wang,
{\em The diffusive logistic equation with a free boundary and sign-changing coefficient},
J. Differential Equations, 258 (2015), 1252-1266.

%


\bibitem{Zla}
A. Zlato\v{s},
{\em Sharp transition between extinction and propagation of reaction}, J. Amer. Math. Soc., 19 (2006), 251-263.

\bibitem{ZKL}
N. Zaker, L. Ketchemen and F. Lutscher,
{\em The effect of movement behavior on population density in patchy landscapes}, Bull. Math. Biol., 82 (2020), 1-24.


\end{thebibliography}
\end{document}